\documentclass[11pt]{amsart}

\usepackage[utf8x]{inputenc}

\usepackage{amsmath, amsthm, amsfonts, amssymb}
\usepackage{epsfig, enumerate}
\usepackage{color}

\usepackage{hyperref}

\DeclareMathOperator{\rank}{{\rm rank}}

\newtheorem{maintheorem}{Theorem} 
\newtheorem{theorem}{Theorem}
\newtheorem{lemma}[theorem]{Lemma}
\newtheorem{coro}[theorem]{Corollary}
\newtheorem{prop}[theorem]{Proposition}

\newtheorem*{assumptions*}{Assumptions}

\newtheorem*{rem*}{Remark}

\theoremstyle{remark}

\newtheorem*{remark*}{Remark}

\theoremstyle{definition}
\newtheorem{definition}{Definition}

\newcommand{\bfa}{{\mathbf a}}
\newcommand{\bfb}{{\mathbf b}}
\newcommand{\bfc}{{\mathbf c}}
\newcommand{\bfd}{{\mathbf d}}

\newcommand{\bfp}{{\mathbf p}}

\newcommand{\Bb}{{\mathcal B}}

\newcommand{\Jj}{{\mathcal J}}

\newcommand{\Mm}{{\mathcal M}}

\newcommand{\Oo}{{\mathcal O}}

\newcommand{\CC}{{\mathbb C}}

\newcommand{\MM}{{\mathbb M}}
\newcommand{\NN}{{\mathbb N}}

\newcommand{\RR}{{\mathbb R}}
\newcommand{\Sb}{{\mathbb S}}

\newcommand{\VV}{{\mathbb V}}
\newcommand{\WW}{{\mathbb W}}
\newcommand{\XX}{{\mathbb X}}

\newcommand{\ZZ}{{\mathbb Z}}

\newcommand{\one}{{\bf 1}}
\newcommand{\nul}{{\bf 0}}

\newcommand{\qtx}[1]{\quad\text{#1}\quad}

\newcommand{\SL}{{\rm SL}}
\newcommand{\GL}{{\rm GL}}

\newcommand{\Un}{{\rm U}}

\newcommand{\pmat}[1]{\begin{pmatrix} #1  \end{pmatrix}}
\newcommand{\smat}[1]{\left( \begin{smallmatrix} #1  \end{smallmatrix} \right)}

\newcommand{\diag}{{\rm diag}}

\DeclareMathOperator{\ran}{{\rm ran}}
\DeclareMathOperator{\Tr}{{\rm Tr}}

\newcommand{\be}[1]{\begin{equation} \label{#1} }
\newcommand{\ee}{\end{equation}}
\newcommand{\beq}{\begin{equation}}

\numberwithin{theorem}{section}
\numberwithin{equation}{section}

\title[Analytic Cocycles with negative infinite Lyapunov exponents]{Singular Analytic Linear Cocycles with negative infinite Lyapunov exponents}

\author{Christian Sadel}

\address[Sadel]{Institute of Science and Technology, 3400 Klosterneuburg, \hbox{Austria} and Facultad de Matem\'aticas, Pontificia Universidad Cat\'olica, Santiago de Chile} 
\email{chsadel@mat.uc.cl}

\author{Disheng Xu}

\address[Xu]{Universit\'e Paris Diderot, Sorbonne Paris Cit\'e, Institut de Math\'ematiques de Jussieu-Paris Rive Gauche, UMR 7586, CNRS, Sorbonne Universit\'es, UPMC Universit\'e Paris 06, F-75013, Paris, France}
\email{disheng.xu@imj-prg.fr}


\begin{document}

\begin{abstract}
We show that linear analytic cocycles where all Lyapunov exponents are negative infinite are nilpotent. For such one-frequency cocycles we show
that they can be analytically conjugated to an upper triangular cocycle or a Jordan normal form. As a consequence, an arbitrarily small analytic perturbation leads to distinct Lyapunov exponents.

Moreover, in the one-frequency case where the $k$-th Lyapunov exponent is finite and the $k+1$st negative infinite, we
obtain a simple criterion for domination in which case there is a splitting into a nilpotent part and an invertible part.
\end{abstract}

\maketitle

\tableofcontents

\section{Introduction}

Let $\XX$ be a compact space, $\mu$ a probability measure on the Borel $\sigma$-algebra of $\XX$ and $f: \XX \to \XX$ a measure preserving transformation, $\mu(f^{-1}(\Bb))=\mu(\Bb)$ for all Borel sets $\Bb\subset \XX$.
Iterations of the map $f$ define a dynamical system on $\XX$, the so called base dynamics.
By $\CC^{d\times d}$ we denote the set of $d\times d$ matrices with complex entries. 
For a measurable map $A: \XX \to \CC^{d\times d}$ one obtains the linear cocycle $(f,A)$ denoting the map
$$
(f,A)\,:\,\XX\times \CC^d\,\to\,\XX\times \CC^d\;,\, (x,v)\,\mapsto\,(f(x),A(x)v)\;.
$$

Some examples of linear cocycles are the derivative cocycle $(f, Df)$ of a $C^1-$ map of torus, the random products of matrices, Schr\"odinger cocycles, etc. 

In general we want to consider analytic cocycles:

\begin{definition}\label{def:1}
We call $(f,A)$ an analytic cocycle over a compact, connected measure space $(\XX,\mu)$ if the following three assumptions hold:
\begin{itemize}
 \item[{\rm (A1)}] $\XX$ is a compact, connected, real analytic manifold
 \item[{\rm (A2)}] For any analytic chart (bi-analytic map) $\varphi:\Oo \subset \XX \to U \subset \RR^\ell$ the push-forward measure $\mu\circ\varphi^{-1}$ 
 on $U$ has a continuous density with respect to the Lebesgue measure on $U$. 
 \item[{\rm (A3)}] $f$ and $A$ are (real) analytic, i.e. $f \in C^\omega(\XX,\XX)$ and $A\in C^\omega(\XX,\CC^{d\times d})$.
\end{itemize}
\end{definition}

Note, if $\XX$ would not be connected then using compactness one finds that a certain iterative power of $f$ would leave the connected components invariant and one could consider the corresponding powers of $(f,A)$ inducing cocycles
on these components.

The prime example we are thinking about are cocycles over the rotation on a torus, i.e. $\XX=\RR^\ell/\ZZ^\ell$, $\mu$ is the canonical Haar measure (or Lebesgue measure), $f(x)=x+\alpha$ with $\alpha\in \RR^\ell/\ZZ^\ell$ and $A\in C^\omega(\RR^\ell/\ZZ^\ell,\CC^{d\times d})$.
Then we may denote the cocycle $(f,A)$ also by $(\alpha,A)$ and call it an $\ell$-frequency cocycle, because the base dynamics is determined by the $\ell$-frequency vector $\alpha$. 

If $\alpha$ is a rational vector, then $A(f^n(x)), n\in\NN$ is a periodic sequence, for $\alpha$ irrational one calls it a quasi-periodic sequence and $(\alpha,A)$ is a quasi-periodic cocycle. 
Such one-frequency quasi-periodic $\SL(2,\RR)$ cocycles have been intensively studied in the past because they are very important for the theory of discrete quasi-periodic one-dimensional Schrödinger operators, see \cite{Av2} and references therein.

For analytic cocycles one often uses some inductive limit topology considering holomorphic extensions\footnote{taking a finite analytic atlas one can technically complexify the arguments $x\in \XX$ in the charts and extend $A(x)$ to a multi-holomorphic function by Taylor expansions} of $\XX$ and $A\in C^\omega(\XX,\CC^{d\times d})$, in the one-frequency case see e.g. \cite{AJS}.

The main object of interest of linear cocycles is the asymptotic behavior of the products of $A$ along the orbits of $f$, especially the \textit{Lyapunov exponents}.
Iterating a linear cocycle leads to $(f,A)^n=(f^n,A_n)$ or $(\alpha,A)^n=(n\alpha,A_n)$, where
\begin{equation}
A_n(x)\,=\,A(f^{n-1}(x))\,A(f^{n-2}(x))\,\cdots\; A(f(x))\,A(x)\;. 
\end{equation}

Let $\sigma_k(A)$ denote the $k$-th singular value of a matrix $A$, i.e. $\sigma_k(A)\geq 0$ and the squares, $\sigma_1^2\geq \sigma_2^2\geq \ldots\geq \sigma_d^2$
are the eigenvalues of $A^*A$. Then, the $k$-th Lyapunov exponent is defined by
\begin{equation}
 L_k(f,A)\,=\,\lim_{n\to\infty}\,\frac1n\, \int_{\XX} \ln(\sigma_k(A_n(x)))\;d\mu(x)\;.
\end{equation}
With $\Lambda^k A$ we denote the linear operator on the anti-symmetric tensor product $\Lambda^k \CC^d$ defined by $\Lambda^k A (v_1\wedge \ldots \wedge v_k)= (Av_1\wedge\ldots \wedge Av_k)$.
Then it is well known that $\prod_{j=1}^k \sigma_j(A)=\|\Lambda^k A\|=\sigma_1(\Lambda^k A)$ giving
\begin{equation}
 \sum_{j=1}^k L_j(f,A)=L_1(f,\Lambda^k A)=\int_\XX \ln \|\Lambda^k A(x)\|\,d\mu(x)\;.
\end{equation}
If we have an $\ell$-frequency cocycle with $f(x)=x+\alpha$, then we may also write $L_k(\alpha,A)$.

\vspace{.2cm}

Let $\int \ln_+\|A(x)\|d\mu(x) <\infty$ where $\ln_+$ is the positive part of the logarithm, then Kingman's Subadditive Ergodic Theorem shows that the Lyapunov exponents exist\footnote{here, $L_k=-\infty$ is possible} with $L_k\in[-\infty,\infty)$.
If $A(x)$ is continuous and always invertible, then all Lyapunov exponents are finite, i.e. bigger than $-\infty$. But if $A(x)$ can have a kernel, then one might end up with some $-\infty$ Lyapunov exponents.
We want to classify these situations for analytic cocycles.

\vspace{.2cm}

Understanding the structure of cocycles is an important branch in the theory of dynamical systems. 
An important question is how frequent cocycles with simple Lyapunov spectrum occur (cf. \cite{GM},\cite{BV},\cite{AV},\cite{GR},\cite{V},\cite{FK},etc.). 
The Lyapunov spectrum is called simple if all Lyapunov exponents are different.
Typically one would expect this to be true on a dense set of cocycles.
This question, however, gets trickier the higher the considered regularity class.
On the other hand, in low regularity ($C^0$), failure of non-uniform hyperbolicity is a fairly robust phenomenon in the topological sense \cite{B}.

For $\SL(2,\RR)$-cocycles Avila showed that the set of cocycles with distinct (or positive) Lyapunov exponents is dense in all usual regularity classes \cite{Av}. 
Distinctness of the largest and smallest Lyapunov exponent on a dense set of general symplectic or pseudo-unitary cocycles of $d\times d$ matrices (in all regularity classes) was shown in \cite{Xu}. 
It relies on Kotani theory and local averaging formulas combining ideas from \cite{Av,AK,KS,Sa}, but a certain {\it real} Lie-group structure is always very important.
For complex analytic $\SL(2,\CC)$ or $\CC^{d\times d}$ cocycles the question is open. 
An approach to distinct Lyapunov exponents has been worked out by Duarte and Klein \cite{DK, DK2} which is based on generalizations of the Avalanche principle and large deviation estimates. These tools had been used a lot for $\SL(2,\RR)$ cocycles (\cite{BJ,Bou,GS}).

Once there is some gap in the Lyapunov spectrum another important concept is that of domination (a generalisation of the notion of uniform hyperbolicity, a precise definition is given below). In \cite{AJS} it was shown that within the set of complex, analytic one-frequency cocycles with a gap in the Lyapunov spectrum, the set of 
dominated cocycles is dense. However, for complex analytic cocycles it is not clear whether the set of cocycles where all Lyapunov exponents are equal has a non-empty interior.

We propose to attack this and further question for complex cocycles by looking for conjugated 'normal forms' similar
as Jordan normal forms or Hilbert-Schmidt decompositions for matrices. 
One should try to classify cocycles where all Lyapunov exponents are equal. 
In this work we consider cocycles where all Lyapunov exponents are negative infinite.
Within the measurable, ergodic category, the Oseledets filtration gives some block upper-triangular normal form, cf. \cite{O,R} which can be refined by looking at so called maximal invariant flags \cite{ACO}.
For invertible cocycles ($f$ and $A$ invertible) one has an Oseledets splitting and a block diagonal normal form. Each block corresponds to a distinct Lyapunov exponent.

Before getting to the normal forms mathematically, we need a proper equivalence relation. Two cocycles $(f,A)$ and $(f,B)$ with the same base dynamics are dynamically conjugated, if
$$
B(x)\,=\,M^{-1}(f(x))\,A(x)\,M(x)
$$
where $M: \XX \to \GL(d)$ is a measurable map into the general linear group. Then, $(f,B)=({\rm id},M)^{-1} (f,A) ({\rm id},M)$ and the cocycles are dynamically equivalent.
However, if $M$ is only measurable and only almost surely defined, 
then one looses regularity features like e.g. analyticity of the cocycle and other certain fine distinctions such as non-uniform and uniform hyperbolicity or the notion of domination.
Therefore, in terms of normal forms we are only interested at dynamical conjugation within the regularity class. Especially in this case we consider {\it analytic} cocycles and we also want $M(x)$ (and hence $B(x)$) to depend analytically on $x$.

\vspace{.2cm}

One way to create cocycles where all Lyapunov exponents are $-\infty$ is by constructing cocycles such that after finitely many steps one arrives at the zero cocycle.
We call such cocycles nilpotent:
\begin{definition}
 A linear cocycle $(f,A)$ is called {\it nilpotent} if for finite $n$ we have $A_n(x)=\nul$ $\mu$-almost surely.
 The minimal such natural number is called the {\it nilpotency degree} $p$.
\end{definition}
Clearly, for nilpotent cocycles all Lyapunov exponents are negative infinite. Our main result is that for analytic cocycles this is an equivalence.
Let us note that in the $C^\infty$ regularity class it is wrong that $L_1=-\infty$ implies nilpotency; even for $1\times 1$ cocycles!
To show this let $A(x)=e^{-1/x^2-1/(1-x)^2}$ for $x\in(0,1)$, $A(0)=A(1)=0$ and continue periodically.
Then, $A\in C^\infty(\RR/\ZZ,\CC^{1\times 1})$, and $(\alpha,A)$ is clearly not nilpotent but 
$$
L_1(\alpha,A) =\,-\,\int_0^1 \frac{1}{x^2} + \frac{1}{(1-x)^2}\; dx\;=\;-\infty\;.
$$

Nilpotency can be achieved by taking upper triangular matrices with zeroes along (and below) the diagonal. 
Our second main result is that in the analytic one-frequency case these are all possibilities up to analytic unitary dynamical conjugation.
Particularly, an arbitrarily small analytic perturbation leads to simplicity of Lyapunov exponents.

If we have only some negative infinite Lyapunov exponents, but $L_1(\alpha,A)>-\infty$ we can split of some nilpotent analytic invariant subspace corresponding to the negative infinite Lyapunov exponents.
In this case we also get some simple criterion for a dominated splitting.

\vspace{.3cm}

In the next section we state the precise theorems and give several remarks.
In Section~\ref{sec:rank1} we treat first the case when the rank of $A(x)$ is at most one and show that $L_1=-\infty$ implies nilpotency. 
Then, based on this result we can treat the case for general rank of $A(x)$ in Section~\ref{sec:rank-r}.
Section~\ref{sec:non-nilpotent} finally considers one-frequency cocycles where only some Lyapunov exponents are negative infinite.
In the Appendix we give some important facts which are used multiple times.

\vspace{.3cm}

\noindent {\bf Acknowledgement:}
This research has been funded by the People Programme (Marie Curie Actions) of the European 
Union's Seventh Framework Programme \hbox{FP7/2007-2013} under REA grant agreement number 291734.
D.X. would like to thank his thesis advisor Artur Avila for the supervision and support, and this research was partially conducted during the period when D.X. visited the Institute of Science and Technology, Austria.

\section{Results}

Having only negative infinite Lyapunov exponents implies nilpotency in the analytic category: 

\begin{maintheorem}\label{main-1}
 Let $(f,A)$ be an analytic cocycle over a compact, connected measure space $(\XX,\mu)$ in the sense of Definition~\ref{def:1} and assume that $L_1(f,A)=-\infty$. 
 Then, $(f,A)$ is nilpotent, more precisely, $A_{r+1}(x)=\nul$ for all $x$, where $r=\max_x \rank A(x) \leq d-1$ is the maximal rank.
\end{maintheorem}

Very concrete normal forms can be found in the one-frequency case. 

\begin{maintheorem}[One frequency case]\label{main-1a}
Let $A\in C^\omega(\RR/\ZZ,\CC^{d\times d})$ and $\alpha\in\RR/\ZZ$ such that $L_1(\alpha,A)=-\infty$. Then the following hold:
 \begin{enumerate}[{\rm (i)}]
 \item There exists a one-periodic analytic function $U\in C^\omega(\RR/\ZZ,\Un(d))$ with values in the unitary group $\Un(d)$, such that $B(x):=U(x+\alpha)^{-1} A(x)\, U(x)$ is upper triangular with zeroes on and below the diagonal.
 More precisely, if the nilpotency degree is $p$ then one can choose $U(x)$ such that $B(x)$ is divided into $p\times p$ blocks (of different size) with upper-triangular block structure,
 \begin{equation}\label{eq-main}
 B(x):=U(x+\alpha)^{-1} A(x)\, U(x)\;=\;\pmat{\nul & D_2(x) & \star & \star \\ & \ddots & \ddots & \star \\ & & \ddots & D_{p}(x) \\ & & & \nul }
 \end{equation}
 \item Assume additionally that for all $n$, $\rank A_n(x)=r_n$ is constant in $x$. Then, there exists a one-periodic analytic function $M\in C^\omega(\RR/\ZZ,\GL(d,\CC))$ such that
 \begin{equation}\label{eq-main-a}
 \Jj:=M(x+\alpha)^{-1} A(x)\, M(x)\;=\;\pmat{J_1 \\ & \ddots \\ & & J_m}
 \end{equation}
 where $m=\dim \ker A(x)$ and
 \begin{equation}\label{eq-main-b}
  J_i\,=\,\pmat {0 & 1 \\ & \ddots & \ddots \\ & & \ddots & 1 \\ & & &  0}
 \end{equation}
\end{enumerate}
 \end{maintheorem}

 Adding a diagonal perturbation $B'=\diag(b_1,\ldots,b_d)$, $|b_j|>|b_{j+1}|$ to $B(x)$ and conjugating it back we obtain the following for
 $A'(x)=U(x+\alpha) B' U(x)^{-1}$ as a corollary of the above theorem.
\begin{maintheorem}
 Let $A\in C^\omega(\RR/\ZZ,\CC^{d\times d})$ and $\alpha\in\RR/\ZZ$ such that $L_1(\alpha,A)=-\infty$, in which case $L_k(\alpha,A)=-\infty$ for all $k=1,\ldots,d$.
 Then, there exists $A'\in C^\omega(\RR/\ZZ,\CC^{d\times d})$ such that for any $\varepsilon\neq0$ all Lyapunov exponents of $(\alpha,A+\varepsilon A')$ are distinct.
 Hence, there are arbitrarily small analytic perturbations with simple Lyapunov spectrum. 
\end{maintheorem}

\begin{rem*} In analogy to {\rm \cite{ACO}} we call $(\alpha,B)$ and $(\alpha,\Jj)$ analytic Jordan normal forms of $(\alpha,A)$. 
As the form $(\alpha,\Jj)$ is much more restrictive, we may call it a completely reduced Jordan normal form.
Let us make some remarks about the existence of such normal forms in the analytic category.
\begin{enumerate}[{\rm (i)}]
\item The condition needed for Theorem~\ref{main-1a}~{\rm (ii)} is satisfied on a dense set of cocycles $(\alpha,A)$ with $L_1(\alpha,A)=-\infty$.
For small enough $t$ one can define $A(x+it)$ by analyticity and a local Taylor expansion. For any $n$ up to the nilpotency degree, there is only a finite set of $(x,t)$ within $[0,1]\times [-\delta,\delta]$ where
the $\rank A_n(x+it)$ is not maximal (equal to $\max_{x} \rank A_n(x+it)$ which is independent of $t$). This follows from analyticity. Hence, for small enough $t$, the cocycle $(\alpha,A(\cdot+it))$ satisfies the condition.
\item For a completely reduced Jordan form as in Theorem~\ref{main-1a}~{\rm (ii)} one may want to relax \eqref{eq-main-b} and allow $\Jj(x)$ and $J_i(x)$ to depend on $x$, where $J_i(x)$ still has only non-zero entries on the superdiagonal\footnote{entries just above the diagonal} which may become $0$ for some $x$. Then the condition that the ranks of $A_{n}(x)$ are constant is not necessary for such a conjugation. However,
$L_1(\alpha,A)=-\infty$ alone is also not sufficient in this case. Examples where such a form can not be reached by (everywhere defined) analytic conjugations are
$$
A(x)=\pmat{ 0 & \cos(2\pi x) & \sin(2\pi x) \\ 0 & 0 & 1 \\ 0 & 0 & 0} \qtx{or} A'(x)=\pmat{0 & 0 & 0 & \cos(2\pi x) \\ 0 & 0 & 1 & 0 \\ 0 & 0 & 0 & \sin(2\pi x) \\ 0 & 0 & 0 & 0}\;.
$$
In both cases the nilpotency degree is 3, $A_3(x)=\nul, \; A'_3(x)=\nul$.
In the first scenario, $\rank A(x)$ is not constant, in the second one, $\rank A'_2(x)$ is not constant.\\
If one allows the conjugation $M(x)$ to be not invertible in finitely many points, then one can always get such a conjugation so that $M(x+\alpha)^{-1} A(x) M(x)=\Jj$ (for almost all $x$).
But as $M^{-1}(x)$ is then not defined for some $x$ (and of course not analytic) it is not an analytic conjugation of the cocycle.

\item In the general analytic case with a higher-dimensional base $\XX$ one can not even necessarily get 'normal forms' like $B(x)$ above by everywhere analytic conjugations.
 One crucial ingredient missing in the general case is an analogue of Lemma~\ref{lem-lift}.
 Let us give an example of an analytic nilpotent 2-frequency $\CC^{2\times 2}$ cocycle that can not be conjugated to such a normal form.
 Let $\alpha=(\alpha_1, \alpha_2)$ be the translation vector for the base dynamics $f(x,y)=(x,y)+\alpha$, $(x,y)\in\RR^2/\ZZ^2$, and let
$$
A(x,y)=\pmat{-\sin(2\pi(x+\alpha_1))\sin (2\pi y) & \sin(2\pi(x+\alpha_1))\sin (2\pi x) \\ -\sin(2\pi(y+\alpha_2))\sin (2\pi y) & \sin(2\pi(y+\alpha_2)) \sin (2\pi x)}
$$
Then we have $A(x,y)$ with rank $1$ almost surely, $A_2(x,y)=\nul$ and the direction of the kernel of $A(x,y)$ (on projective space $P\CC^2$) has no limit at $(x,y)=(0,0)$ or $(x,y)=(\frac12,\frac12)$ which 
contradicts with analyticity of $M(x,y)$ to get 
$$
[M(x+\alpha_1,y+\alpha_2)]^{-1} A(x,y) M(x,y)=\pmat{0 & c(x,y) \\ 0 & 0}\;.
$$
One may choose $M(x,y)=\pmat{\sin(2\pi x) & -\sin(2\pi y) \\ \sin(2\pi y) & \sin(2\pi x)}$ for conjugating to such a normal form, however, the inverse of $M(x,y)$ does not exist at $(x,y)=(0,0)$ or $(x,y)=(\frac12,\frac12)$.
\end{enumerate}
\end{rem*}

Next, we have a look at analytic one-frequency cocycles where some but not all Lyapunov exponents are $-\infty$. In this case there is an obvious gap after the last finite Lyapunov exponent and one can ask the question about domination. In general this was classified in \cite{AJS}, however, in this special case the classification is much simpler. For completeness let us repeat the definition of domination.
Let $G(k,d)$ denote the Grassmannian manifold of complex $k$-dimensional subspaces of $\CC^d$ and $C(\XX,G(k,d))$ the set of continuous functions from $\XX$ to $G(k,d)$.

\begin{definition}
A continuous $\CC^{d\times d}$ cocycle $(f,A)$ over $(\XX,\mu)$ is $k$-dominated ($k<d$) if there is a continuous splitting
of the space $\CC^d$ in a relatively stable and a relatively unstable invariant space, i.e. there exist $u\in C(\XX,G(k,d))$, $s\in C(\XX,G(d-k,d))$ such that
for all $x\in \XX$,
$$
\CC^d=u(x)\oplus s(x)\,,\quad A(x) u(x)=u(f(x))\,,\quad A(x) s(x) \subset s(f(x))
$$
and for some $n\in \NN$ and all $0\neq v \in u(x),\,0\neq w \in s(x)$ and all $x\in \XX$ one has
$$
\|A_n(x) v\,/\,\|v\|\;>\; \|A_n(x) w\|\,/\,\|w\|\;. 
$$
Particularly, the kernel is always inside the relatively stable space, $\ker A(x)\subset s(x)$.
\end{definition}

\begin{maintheorem}\label{main-2} 
\vbox{Let $A\in C^\omega(\RR/\ZZ,\CC^{d\times d})$ and $\alpha\in\RR/\ZZ$ such that $L_{k+1}(\alpha,A)=-\infty$ and $L_k(\alpha,A)>-\infty$. Then the following hold.
\begin{enumerate}[{\rm (i)}]
 \item  There exists
 $U\in C^\omega(\RR/\ZZ,\Un(d))$ such that 
 $$
 B(x)\,:=\,U(x+\alpha)^{-1} A(x)\, U(x)\;=\;\pmat{\bfa(x) & \bfb(x) \\ \nul & \bfd(x) } \qtx{where} \bfa(x)=\pmat{0 & \star & \star \\ & \ddots & \star \\ & & 0}
 $$ 
 is an upper triangular $(d-k)\times (d-k)$ matrix with zeros on and below the diagonal, hence $(\alpha,\bfa)$ is nilpotent and
 $\bfd(x)$ is an almost surely invertible $k\times k$ matrix. In particular, $\rank A_{d-k}(x)\leq k$ for all $x$. Of course, the block $(\alpha,\bfa)$ can also be conjugated to a Jordan form by an analytic dynamical conjugation as described above.
 \item The cocycle $(\alpha,A)$ is $k$-dominated if and only if $\rank A_{d-k}(x) =k$ for all $x$.
 It is also equivalent to $\bfd(x)$ as defined in {\rm (i)} being invertible for all $x\in \RR/\ZZ$. 
In this case there is some analytic $(d-k)\times k$ matrix $M(x)$ such that with $B(x)$ as in {\rm (i)} we have
 $$
 C(x)\,:=\,\pmat{\one & M(x+\alpha) \\ \nul & \one}^{-1}\,B(x)\,\pmat{\one & M(x) \\ \nul & \one}\,=\,\pmat{\bfa(x) & \nul \\ \nul & \bfd(x)}
 $$
This conjugation corresponds to the dominated splitting.
\end{enumerate}}
\end{maintheorem}

\begin{rem*}
Without domination it is not always true that one can obtain this block-diagonal form with an analytic (or everywhere defined) conjugation.
A counter-example is the following cocycle, $A(x)=B(x)=\smat{0 & \cos(2\pi x) \\ 0 & \sin(2\pi x)}$ with any frequency $\alpha\in\RR/\ZZ-\{0\}$.
A diagonal, analytic conjugated cocycle would necessarily be of the form $C(x)=M^{-1}(x+\alpha) A(x) M(x)=\smat{0 & 0 \\ 0 & c(x)}$.
As $A_2(0)=A(\alpha)A(0)=\nul$ one has $c(0)=0$ or $c(\alpha)=0$, i.e. $C(x)=\nul$ for either $x=0$ or $x=\alpha$. But, $A(x)=M(x+\alpha) C(x) M(x)^{-1} \neq \nul$ 
for any $x$. So there is a contradiction if $M(x)$ is invertible for all $x$.

\end{rem*}

\section{Rank one case \label{sec:rank1}}

In this section we will basically prove Theorem~\ref{main-1} and Theorem~\ref{main-1a} in the rank one case by the following Proposition.

\begin{prop}\label{prop-rank-1} We have the following.
\begin{enumerate}[{\rm (i)}]
\item Assume that $(f,A)$ is an analytic cocycle over a compact and connected space $(\XX,\mu)$ as defined in Definition~\ref{def:1}. 
Assume further that $A\in C^\omega(\XX,\CC^{d\times d})$ has maximal rank $1$ and $L_1(\alpha,A)=-\infty$. Then, $A_2(x)=\nul$ for all $x \in \XX$.
\item Let $(\alpha,A)$ be an analytic one-frequency cocycle, i.e. $\alpha\in\RR/\ZZ$, $A\in C^\omega(\RR/\ZZ,\CC^{d\times d})$, and let $L_1(\alpha,A)=-\infty$.
Then, there is a one-periodic analytic function $c(x)$ and a one-periodic analytic unitary function $U(x)\in\Un(d)$ such that
$$
U^*(x+\alpha) \,A(x)\, U(x)\,=\,\pmat{ \nul & \\ & 0 & c(x) \\ & 0 & 0 }\;.
$$
\end{enumerate}
\end{prop}

\begin{proof}
We will first show (ii). 
The case $A(x)=\nul$ for all $x$ is trivial, so assume $A(x)\neq \nul$ for some $x$.
 We find some column-vector $\varphi(x)$ of $A(x)$ which is not always zero. By Lemma~\ref{lem-lift} we find a one-periodic, real analytic 
function $\phi(x)$ with $\|\phi(x)\|=1$ such that $\varphi(x)$ is a complex multiple of $\phi(x)$. 
Doing the same with $A^*(x)$ we obtain some one-periodic analytic function $\psi(x)$ with $\|\psi(x)\|=1$. As $\ran A(x)\subset \phi(x)\CC,\,\ran A^*(x)\subset \psi(x)\CC$ (at most rank 1), we find some $c(x)$ such that
\begin{equation}\label{eq-rank-1}
A(x)\,=\,c(x)\,\phi(x)\, \psi^*(x)
\end{equation}
where $\phi(x)$ is a column-vector and $\psi^*(x)$ a row vector and their product a matrix. As $A(x)$ depends analytically on $x$, $c(x)$ has to be analytic.
Thus,
$$
A_n(x)=\left( \prod_{k=0}^{n-1} c(x+k\alpha) \right) \left( \prod_{k=0}^{n-2} \psi^*(x+(k+1)\alpha) \phi(x+k\alpha) \right)\, \phi(x+(n-1)\alpha) \,\psi^*(x)
$$
which leads to
\begin{equation}\label{eq-L_1-rank1}
-\infty\,=\,L_1(\alpha,A)\,=\,\int_0^1 \ln|c(x)|\,dx\,+\,\int_0^1 \ln| \psi^*(x+\alpha) \phi(x)|\,dx\;.
\end{equation}
By Lemma~\ref{lem-f-zero} this implies $\psi^*(x+\alpha) \phi(x)=0$ for all $x$ as $c(x)$ is not the zero function.
This gives $A_2(x)=\nul$. Moreover by Lemma~\ref{lem-lift}~(ii) one can extend $\phi(x-\alpha),\psi(x)$ to an orthonormal basis\footnote{Indeed this task is equivalent in finding $\Theta(x)$ as in Lemma~\ref{lem-lift}~(ii) where the range of $\Theta(x)$ is the orthogonal complement
of the space spanned by $\phi(x-\alpha)$ and $\psi(x)$.} defining a unitary matrix
$U(x)=(\Theta(x),\phi(x-\alpha),\psi(x))$ such that $U^*(x+\alpha) \,A(x)\, U(x)\,=\,\smat{\nul \\ & 0 & c(x) \\ & 0 & 0  }$\;.

In the general case (i) we still find functions $c(x)$ $\phi(x)$, $\psi(x)$ with $\|\phi(x)\|=\|\psi(x)\|=1$ satisfying \eqref{eq-rank-1}. However, we can only guarantee analyticity at points $x$ where $\rank A(x)=1$. 
In general, there might be some union of sub-manifolds of $\XX$ where $A(x)=\nul$, (i.e. $c(x)=0$) and where $c(x), \phi(x), \psi(x)$ may not be analytic. But the functions
\begin{align}
&g_1(x)\,:=\, \Tr(A(x)^*A(x))\,=\,|c(x)|^2  \\
&g_2(x)\,:=\,\Tr(A_2(x)^* A_2(x))\,=\, |c(x)c(f(x))|^2\,|\psi^*(f(x)) \phi(x)|^2
\end{align}
are always analytic. We assume again that $A(x)$ is not identically zero, in which case $g_1(x)$ is not identically zero.
Then, similar to \eqref{eq-L_1-rank1} we find
$$
-\infty\,=\,L_1(f,A)\,=\,\frac12\,\int_\XX\, \ln(g_2(x))\,-\,\ln(g_1(x))\,d\mu(x)\;.
$$
Using Lemma~\ref{lem-f-zero} we find that $g_2(x)=0$ for all $x$, but this is equivalent to $A_2(x)=\nul$ for all $x\in\XX$.
\end{proof}

\section{General rank case \label{sec:rank-r}}

We start with the following simple observation:
\begin{lemma}\label{lem-rank}
Assume $(f,A)$ is an analytic cocycle over a compact, connected  measure space $(\XX,\mu)$. 
There is $r$ such that for all $x$ except a union of sub-manifolds of zero measure (w.r.t. $\mu$), $\rank (A(x))=r$ and $\rank(A(x))\leq r$ for all $x$.
\end{lemma}
\begin{proof} 
Let $r=\max_{x} \rank A(x)$, such that $\rank A(x)\leq r$ for all $x$. Then, $\Lambda^r A(x) \neq \nul$ for some $x$ and the equation $\rank A(x)<r$ is equivalent to $\Lambda^r A(x)=\nul$. 
By analyticity and connectedness of $\XX$, in any chart for $\XX$, the equation $\Lambda^r A(x) = \nul$ defines a union of sub-manifolds of zero Lebesgue measure within the chart (see also Corollary~\ref{cor-f-zero}). 
Using a finite atlas for $\XX$ and Assumption (A2) in Definition~\ref{def:1} gives the claim.
\end{proof}
\noindent Note that in the one-dimensional case $\XX=\RR/\ZZ$, this zero-measure set simply consists of finitely many points.

\vspace{.2cm}

Another special point of analytic cocycles is the fact that the rank reduction has to take place in each step:

\begin{lemma}\label{lem-rank-reduction}
Let $(f,A)$ denote an analytic cocycle over a compact, connected measure space $(\XX,\mu)$ such that for some $m>0$ and all $x\in\XX$ we have $\rank(A_n(x))\leq r$ and   $\rank(A_{n+m}(x))<r$. 
Then $\rank(A_{n+1}(x))<r$ for all $x\in\XX$. 
\end{lemma}
\begin{proof}
We claim if $m>1$, then $\rank(A_{n+m-1}(x))<r$ for all $x$. The result then follows by backward induction.
We let
$$\Bb:=\{x:\rank(A_{n+m-1}(x))<r\}\,=\,\{x:\Lambda^r A_{n+m-1}(x)=\nul \}\,.$$
Take some $x\not \in \Bb$, then 
\begin{equation*}
	\rank(A_{n+m}(x))< r\,,\quad
\rank(A_{n+m-1}(x))= r\,.
\end{equation*}
As $A_{n+m}(x)=A(f^{n+m-1}(x))\,A_{n+m-1}(x)$ this means $\ran A_{n+m-1}(x)\cap \ker A(f^{n+m-1}(x))\neq \emptyset$.
Since $A_{n+m-1}(x)=A_{n+m-2}(f(x)) A(x)$ and $m>1$ we find $r\geq\rank A_{n+m-2}(f(x))\geq r$ and hence
$$\ran(A_{n+m-2}(f(x))=\ran(A_{n+m-1}(x))\,.$$ 
Therefore, $$\ran A_{n+m-2}(f(x))\cap \ker A(f^{n+m-1}(x))\neq \emptyset$$
implying  $$\rank(\,A_{n+m-1}(f(x))\,)<r\qtx{which means} f(x) \in \Bb $$
In summary, we prove that for all $x\in \XX$, either $x\in \Bb$ or $f(x)\in \Bb$, i.e. $\Bb\cup f^{-1}(\Bb)=\XX$. 
This implies $\mu(\Bb)=\mu(f^{-1}(\Bb))>0$ and by Corollary~\ref{cor-f-zero}, $\Lambda^r A=\nul$ for all $x$ and hence $\Bb=\XX$.
\end{proof}

Now we can prove Theorem~\ref{main-1} and Theorem~\ref{main-1a}.

\begin{proof}[Proof of Theorem~\ref{main-1}]
Let $(f,A)$ denote an analytic cocycle over a compact, connected measure space $(\XX,\mu)$ and $L_1(f,A)=-\infty$.
By Lemma~\ref{lem-rank} we know for any $n\in \NN$, there is $r_n$ such that for all $x$ except a $\mu$-zero-measure set $\rank(A_n(x))=r_n$ and $\rank(A_n(x))\leq r_n$ for all $x$. Then we have $r_{n-1}\geq r_n$ for all $n$. 
Let $\tilde r=\min_{n\in\NN}r_n$, to prove the Lemma we need to establish $\tilde r=0$. Suppose $\tilde r=r_n>0$. Therefore, $\Lambda^{r_n} A_n(x)$ has maximal rank $1$ and 
$$L_1(f^n, \Lambda^{r_n}(A_n))=n\,\sum_{i=1}^{r_n} L_i(f,A)=-\infty$$
By Proposition~\ref{prop-rank-1} we have $\Lambda^{r_n} A_{2n}(x)=\nul$ for all $x$. As a result, $r_{2n}=\rank(A_{2n}(x))<r_n$ which contradicts with our assumption of $r_n=\tilde r$.
Iterating Lemma~\ref{lem-rank-reduction} gives $A_{r+1}(x)=\nul$ for all $x\in\XX$ with $r=r_1$.
\end{proof}

\begin{proof}[Proof of Theorem~\ref{main-1a}]
Now let $(\alpha,A)$ be an analytic one-frequency cocycle and $L_1(\alpha,A)=-\infty$. By the proof above we know that $(\alpha,A)$ is nilpotent.
Let $p$ be the nilpotency degree. As a corollary of the lemma above we get that $\ker A_n(x)$ is strictly increasing,
$\ker A_{n-1}(x) \subsetneqq \ker A_n(x)$ for $n=1,\ldots p$ and almost all $x$, hence $\rank A_n(x)=r_n$ and the kernels have dimensions $d-r_n$.
Note that by Lemma~\ref{lem-an-sub} the subspaces $(\ker A_{n-1}(x))^\perp \cap \ker A_n(x)$ induce an analytic function 
from $\RR/\ZZ$ to $G(r_{n-1}-r_n,d)$.
Using Lemma~\ref{lem-lift}~(ii) this means
that we find analytic dependent matrices $M_n(x)\in \CC^{d\times(r_{n-1}-r_n)}$, $n=1,\ldots,p$ such that:\\
(i) $M_n(x)^* M_n(x)=\one$ for all $x$\\
(ii) $\ran M_n(x)$ is orthogonal to the kernel of $A_{n-1}(x)$ for almost all\footnote{Here, almost all means all but finitely many} $x$\\
(iii) the range of $M_n(x)$ and the kernel of $A_{n-1}(x)$ span the kernel of $A_n(x)$, for almost all $x$.\\
Here, $A_0(x)=\one$ and so $M_1$ actually spans the kernel of $A(x)$.

\vspace{.2cm}

As $r_n=0$ we get $\ker A_n = \CC^d$ and hence $U(x)= (M_1(x),\ldots,M_k(x))$ defines an analytic unitary matrix.
As $\ker A_n(x)=\bigoplus_{i=1}^n \ran M_i$ for $n=1,\ldots,p$ and almost all $x$ we obtain that $B(x):=U^*(x+\alpha) A(x) U(x)$ is of the claimed form \eqref{eq-main}, at first, for almost all $x\in \RR/\ZZ$, but by analyticity for all $x$. This shows part (i).

\vspace{.2cm}

Now let us get to the completely reduced Jordan form, part (ii). 

We assume that $\rank A_n(x)=r_n$ is constant for all $x$. 
Recall that the nilpotency degree of $A$ was denoted by $p$. By Lemma~\ref{lem-an-sub}
the subspaces $\VV_n$ 
$$
\VV_n(x)\,:=\,\ran(A_{p-n}(x-(p-n)\alpha))\,=\, \ran (A(x-\alpha)\cdots A(x-(p-n)\alpha)\;,
$$
of fixed dimensions $r_{p-n}$, $n=1,\ldots,p$ are analytically dependent on $x$, where we set $\VV_p=\CC^d$.
Clearly, $\VV_n(x) \subset \VV_{n+1}(x)$ and $A(x) \VV_{n}(x) = \VV_{n-1}(x+\alpha) \subset \VV_n(x+\alpha)$.
Choosing some analytically dependent basis of $\VV_{n}(x)$ it is clear that $A|_{\VV_n}(x)$ defined as $A(x)$ mapping from $\VV_{n}(x)$ to $\VV_{n}(x+\alpha)$ is analytic and by assumption of constant rank $r_{p-n}$.
Thus, by Lemma~\ref{lem-an-sub}, the subspaces $\ker A|_{\VV_n}(x) \subset \VV_{n}(x)$ and their orthogonal complements within $\VV_n(x)$, $(\ker A|_{\VV_n}(x))^\perp$ depend analytically on $x$.
Then, by constancy of the rank, the restriction of the map $A(x)$ (or $A|_{\VV_n}(x)$) from  $(\ker A|_{\VV_n}(x))^\perp$ to $\VV_{n-1}(x+\alpha)$ is analytic and invertible for all $x$.
Taking the inverse, we get some analytic function $\hat A_{\VV_{n-1}}(x)$ such that
\begin{equation}\label{eq-def-hatA}
\hat A_{\VV_{n-1}}(x)\,:\,\VV_{n-1}(x+\alpha)\,\to\,\VV_n(x),\quad   A(x)\,\hat A_{\VV_{n-1}}(x) v\,=\,v\qtx{for}  v\in \VV_{n-1}(x+\alpha)\;.
\end{equation}

We claim that for any $1\leq n\leq p$, there exists analytic maps $v_{i,j}:\RR/\ZZ\to \CC^{d-1}\setminus\{0\}$, $1\leq i\leq l,  1\leq j\leq d_i, $ such that 
\begin{align}
& A(x)\, v_{i,j}(x) \,=\,  v_{i,j+1}(x+\alpha),\text{ where } v_{i,d_i+1}:=0 \label{jor claim eq1}\\
& \{v_{i,j}(x)\}_{1\leq i \leq l, 1\leq j\leq d_i} \quad \text{is a linear independent family for {\bf all} $x$} \label{jor claim eq2}\\
& \VV_n(x)\,=\, \label{jor claim eq3}\;{\rm span}\{v_{i,j}(x)\}_{1\leq i \leq l, 1\leq j\leq d_i}\quad\text{for {\bf all} $x$}\;.
\end{align}
The values $l$ and $d_i$ depend on $n$. Notice that to prove the existence of a Jordan form we only need to prove the claim for the case $n=p$.

\vspace{.2cm}

We prove the claim by induction: when $n=1$, $\VV_1(x)\subset \ker(A(x))$.
By Lemma \ref{lem-lift} and Appendix of \cite{AJS}, there are analytic maps $v_{i,1}:\RR/\ZZ\to \CC^{d-1}$ such that 
for any $x$, $(v_{i,1}(x))_{i=1}^\ell$ is a basis of $\VV_1(x)$
which proves the claim for the case $n=1$.

Suppose the claim holds for $n-1<p$, i.e. there are analytic maps $v_{i,j}$ satisfying \eqref{jor claim eq1}, \eqref{jor claim eq2}, \eqref{jor claim eq3}. 
Then $v_{i,1}(x+\alpha) \in \VV_{n-1}(x+\alpha)$  and using the analytic dependent maps $\hat A_{\VV_{n-1}}(x)$ as in \eqref{eq-def-hatA} we can define the analytic vectors
$$
v_{i,0}(x)\,:=\,\hat A_{\VV_{n-1}}(x)\,v_{i,1}(x+\alpha)\,\in\,\VV_n(x)\;.
$$
By construction, $\VV_n(x)\subset A(x)^{-1} \VV_{n-1}(x+\alpha)$ where the inverse denotes the pre-image.
By assumption, the latter pre-image is spanned by $\VV_{n-1}(x)$ and $A(x)^{-1}(v_{i,1}(x+\alpha))= v_{i,0}(x)+\,\ker A(x)$.
Hence, $\VV_n(x)$ is spanned by $\VV_{n-1}(x)$, the vectors $v_{i,0}(x)$ and some vectors in $\VV_n(x) \cap \ker A(x)$.

Now, by constancy of $\rank A|_{\VV_n}(x) = \rank A_{p+1-n}(x-(p-n)\alpha)$ we get that the dimensions of $\ker A(x) \cap \VV_n(x) = \ker A|_{\VV_n}(x)$ are constant.
Hence, the orthogonal complement $\WW_n(x)$  of $\ker A(x) \cap \VV_{n-1}(x)$ 
within $\ker A(x) \cap \VV_n(x)$ has constant dimension and is an analytically dependent subspace.
Using Lemma~\ref{lem-lift}~(ii) we find analytic functions $v_{i,1}(x),\,l<i\leq l'$ such that for all $x$,
$$ 
\{v_{i,1}(x) \}_{l < i\leq l'} \qtx{is a basis of} \WW_n(x)\,. 
$$
Moreover, by the considerations above, $\VV_n(x)$ is spanned by $\VV_{n-1}(x)$, $ \{v_{i,0}(x)\}_{i=1}^l$ and $\WW_n(x)$
for all $x\in\RR/\ZZ$.
We claim that $\VV_{n-1}(x)$, $\{v_{i,0}(x)\}_{1\leq i \leq l}$ and $\{v_{i,1}(x)\}_{l<i\leq l'}$ are linear independent.
Assume 
$$
v(x)=\sum_{i\leq l} a_i v_{i,0}(x)+\sum_{i>l} a_i v_{i,1}(x) \in \VV_{n-1}(x)\;.
$$ 
Then apply $A(x)$ to get $\sum_{i\leq l} a_i v_{i,1}(x+\alpha) \in A(x)\VV_{n-1}(x)$ which by induction assumption is spanned by $v_{i,j}(x+\alpha)$ with $j\geq 2$.
Hence, $a_i=0$ for all $i\leq l$ by linear independence of $(v_{i,j}(x+\alpha))_{i\leq l,j\geq 1}$.
Thus, $\sum_{i>l} a_i v_{i,1}(x) \in \VV_{n-1}(x)$. By construction, the space $\WW_n(x)$ is transversal to $\VV_{n-1}(x)$ and $\{v_{i,1}(x)\}_{i>l}$ is a basis of $\WW_n(x)$. Hence, $a_i=0$ also for $i>l$, showing the linear independence.

In summary, let 
$$d_i':=\begin{cases}
d_i+1 \text{ if }i\leq l\\ 
1 \text{ if } l<i\leq l'
\end{cases}$$ and $$u_{i,j}(x):= \begin{cases}
v_{i,j-1}(x) \text{ if }i\leq l\\
 v_{i,1}(x) \text{ if }l<i\leq l'
\end{cases}$$
Then $\{u_{i,j}(x)\}_{1\leq i\leq l', 1\leq j\leq d'_i}$ satisfy \eqref{jor claim eq1}, \eqref{jor claim eq2}, \eqref{jor claim eq3} for $n$. 
By induction the claim holds for all $1\leq n\leq p$. 
\end{proof}

\section{Non-nilpotent case \label{sec:non-nilpotent}}

In this section we prove Theorem~\ref{main-2}.
Let us now assume that $L_k(\alpha,A)>-\infty$ and $L_{k+1}(\alpha,A)=-\infty$, $1\leq k <d$.
Let $r_n$ be the maximal rank of $A_n(x)$ for $x\in \RR/\ZZ$, as in Lemma~\ref{lem-rank-reduction}.
As $L_1(\alpha,\Lambda^k A)>-\infty$, we have $\min r_n\geq k$.
Since $L_1(\alpha,\Lambda^{k+1} A)=-\infty$ we know by Theorem~\ref{main-1} that $\Lambda^{k+1} A$ is nilpotent. Hence, for some $n$, $\rank A_n(x)\leq k$ for all $x$ and $\min_n r_n=k$.
By Lemma~\ref{lem-rank-reduction} the rank reduces at every step.

Now, let $p$ be the minimal natural number such that $r_p=k$. 
Then, using Lemma~\ref{lem-an-sub}, $\ker A_p(x)$ induces a $d-k$ dimensional, analytically dependent, invariant subspace.
Let $M_1(x)\in \CC^{d\times{d-k}}$ be an analytic partial isometry such that the column vectors span $\ker A_p(x)$ (almost surely), constructed by Lemma~\ref{lem-lift}.

Again, by Lemma~\ref{lem-an-sub} the orthogonal complement $(\ker A_{p}(x))^\perp$ induces a $k$-dimensional analytically dependent subspace and by Lemma~\ref{lem-lift}~(ii)
we can construct an analytic partial isometry $M_{2}(x)\in \CC^{d\times k}$ where the column vectors span this space.
Then, $U(x)=(M_1(x), M_2(x))$ is by construction an analytically dependent, unitary matrix and we get the desired form
$$
B(x)\,:=\,U^*(x+\alpha)\, A(x)\,U(x)\,=\,\pmat{\bfa(x) & \bfb(x) \\ \nul & \bfd(x)}
$$
where $(\alpha,\bfa)$ is a nilpotent cocycle, $\bfa_p=\nul$, and $\bfd(x)$ is almost surely invertible. 

The fact that this cocycle is $k$-dominated if and only if $\det \bfd(x)\neq 0$ for all $x$ follows directly from the theory in \cite{AJS}.
But it can also be seen more directly. Clearly, if $(\alpha,A)$ and hence also $(\alpha,B)$ is $k$-dominated then using $\bfa_{d-k}=\nul$ and $L_k(\alpha,A)>-\infty$ one must have 
that $\bfd_{d-k}$ is invertible for all $x$ which also implies  $\rank A_{d-k}(x)=\rank B_{d-k}(x)=k$ for all $x$.

Let us now assume $\bfd(x)$ is invertible for all $x$ and construct the dominated splitting. 
We will consider an iteration of dynamical conjugations by $\Mm_n(x)=\smat{\one & M_n(x) \\ \nul & \one}$ which are inductively defined.
Let $C^{(0)}(x)=B(x)$, $\bfc^{(0)}(x)=\bfb(x)$ and define inductively $M_{n+1}(x)=\bfc^{(n)}(x-\alpha)\bfd^{-1}(x-\alpha)$
and $\bfc^{(n+1)}(x)=\bfa(x) M_{n+1}(x)=\bfa(x)\bfc^{(n)}(x-\alpha)\bfd^{-1}(x-\alpha)$.
Then, induction yields
$$
C^{(n)}(x)\,:=\,\Mm_n^{-1}(x+\alpha)\,C^{(n-1)}(x)\,\Mm_n(x)\,=\, \pmat{\bfa(x) & \bfc^{(n)}(x) \\ \nul & \bfd(x)}{}
$$
Note that $\bfc^{(n)}(x)=
\bfa_n(x-(n-1)\alpha)\bfp(x)$ for some matrix $\bfp(x)$. As $(\alpha,\bfa)$ is nilpotent, $\bfa_p=\nul$, this means that $\bfc^{(p)}(x)=\nul$.
Taking $M(x)=\sum_{n=1}^p M_n(x)$ we get
$$
C^{(p)}(x)\,=\,\pmat{\one & M(x+\alpha) \\ \nul & \one}^{-1} \,B(x)\,\pmat{\one & M(x+\alpha) \\ \nul & \one}\,=\,\pmat{\bfa(x) \\ & \bfd(x)}\;.
$$
It is clear that this dynamical conjugation corresponds to a $k$-dominated splitting (as $\bfd(x)$ is always invertible and $\bfa(x)$ nilpotent).
This finishes the proof of Theorem~\ref{main-2}.
\hfill $\Box$

\appendix

\section{Some Lemmata}

\subsection{Lifting lemma and analytic subspaces}

We denote the set of $k$ dimensional subspaces of $\CC^d$ by $G(k,d)$. $G(k,d)$ is a compact holomorphic manifold, $G(1,d)$ is equal to the complex projective space $P\CC^d$.

\begin{lemma} \label{lem-lift}  
We have the following:
\begin{enumerate}[{\rm(i)}]
 \item Every non-zero one-periodic real analytic function $\varphi: \RR/\ZZ \to \CC^d$, induces a real analytic function $\Phi\,:\RR/\ZZ\,\to\, P\CC^d$ to the projective space, such that\footnote{elements in projective space are considered as 1-dimensional subspaces of $\CC^d$}
 $\varphi(x)\in \Phi(x)$. \\
 Every analytic function $\Phi\,:\,\RR/\ZZ\,\to\,P\CC^d$ can be lifted to a one-periodic analytic function $\phi:\RR/\ZZ \to \Sb_\CC^{d-1}$, the set of unit vectors in $\CC^d$, i.e. $\phi(x)\in \Phi(x)$.
 \item Every real analytic function $M: \RR/\ZZ \to \CC^{d \times k}$ with $\sup_x \rank M(x)=k$ induces a real analytic function $\MM: \RR/\ZZ \to G(k,d)$ such that $\ran M(x)\subset\MM(x)$.\\
 Every real analytic function $\MM: \RR/\ZZ \to G(k,d)$ can be lifted to a one-periodic analytic function $\Mm: \RR/\ZZ \to \CC^{d\times k}$ with $\Mm(x)^* \Mm(x)\,=\,\one_k$, i.e. the column vectors of $\Mm(x)$ form an analytically dependent orthonormal basis of $\MM(x)$. 
 
 \end{enumerate}
\end{lemma}
\begin{proof}
For part (i) if $\varphi(x)\neq 0$ then the equivalence class $[\varphi]_\sim$ in projective space is analytic. 
The problematic points are only the values $x_0$ where $\varphi(x_0)=0$. Around such a point $\varphi(x_0+\varepsilon)=\varepsilon^m \hat\varphi(x_0+\varepsilon)$ where $m\in\NN,\,\hat\varphi(x_0)\neq 0$ and $\hat\varphi$ is analytic.
The equivalence class $[\hat\varphi]_\sim$ gives the analytic extension to get $\Phi(x)\in P\CC^d$.
As shown in \cite[Appendix, Theorem~A.1~(vi)]{AJS} for any such function $\Phi$ there is an analytic lift to a one-periodic, non-zero function, normalizing its norm gives $\phi(x)$.\\
For part (ii) note that $G(k,d)$ is a closed sub-manifold of the projective space $P\Lambda^k \CC^d$ by identifying the subspace spanned by $v_1,\ldots, v_k$ with the vector
$v_1 \wedge v_2\wedge \ldots \wedge v_k$.
Thus let $v_1,\ldots,v_k$ be the column vectors of $M$ and use part (i) and closedness of $G(k,d)\subset P\Lambda^k \CC^d$ to get the one-periodic real analytic function $\MM(x)\in G(k,d)\subset P\Lambda^k \CC^d$. 
Again, following \cite[Theorem~A.1~(vi)]{AJS} we get some lift to a function $\hat \Mm(x)\in \CC^{d\times k}$ which has always full rank. Applying the Gram-Schmidt procedure gives $\Mm(x)$.
\end{proof}
One may note that $\phi(x)=e^{if(x)} \varphi(x)/\|\varphi(x)\|$ for some adequate real valued function $f$ in case (i). The proof uses very much the one-dimensional structure of $\RR/\ZZ$ as well as the complex structure of $\CC^d$ or $G(k,d)$. 
The statements are not valid for a higher dimensional base, e.g. $\RR^\ell/\ZZ^\ell$, or when using real Grassmannian manifolds, like $P\RR^d$ instead of $P\CC^d$.

\vspace{.3cm}

Next we consider analytic dependent subspaces. We say that $\MM(x)$, $x\in\RR/\ZZ$ is an analytic subspace if $\MM \in C^\omega(\RR/\ZZ,G(k,d))$.
We say that a family of subspaces $\VV(x)$ induces an analytic subspace if there exists $k$ and 
$\MM\in C^\omega(\RR/\ZZ,G(k,d))$ such that $\VV(x)=\MM(x)$ for almost all $x$.

\begin{lemma}\label{lem-an-sub}
 Let $A(x)$ be an analytic matrix and $\VV(x)$ and $\WW(x)$ analytic subspaces, i.e. 
 $A\in C^\omega(\RR/\ZZ,\CC^{d\times d})$, $\VV\in C^\omega(\RR/\ZZ,G(k,d))$, $\WW\in C^\omega(\RR/\ZZ,G(k',d))$. Then we have
 \begin{enumerate}[{\rm (i)}]
  \item The image $A(x) \VV(x)$ induces an analytic subspace. If $\rank A(x)\VV(x)$ is constant then it {\bf is} an analytic subspace. 
  Particularly, $\ran A(x)$ induces an analytic subspace.
  \item The orthogonal projections $P(x), Q(x)$ onto $\VV(x)$ and $\VV(x)^\perp$ are analytic. Particularly, $\VV(x)^\perp$ is an analytic subspace.
  \item The pre-image $A(x)^{-1} \VV(x)$ induces an analytic subspace and it is an analytic subspace if it has constant dimension. 
  Particularly, $\ker A(x)$ induces an analytic subspace.
  \item The sum $\VV(x)+\WW(x)$ induces an analytic subspace and it is analytic if it has constant dimension.
  \item The intersection $\VV(x)\cap \WW(x)$ induces an analytic subspace and it is analytic if it has constant dimension.
 \end{enumerate}
\begin{proof}
 We let $V(x)$ and $W(x)$ be analytic $d\times k$ and $d\times k'$ matrices such that the column vectors form an orthonormal basis of $\VV(x)$ and $\WW(x)$, respectively.
 These marices exist by Lemma~\ref{lem-lift}.
 Note that $A(x)V(x)$ is an analytic $d\times k$ matrix, choosing column vectors forming a basis of the range for almost all $x$
 and using Lemma~\ref{lem-lift} shows (i).
 Let $v_i(x)$ be the column vectors of $V(x)$ and $P(x)=\sum v_i(x) v_i^*(x)$ the analytic orthogonal projection onto $\VV(x)$.
 Then $Q(x)=\one-P(x)$ is analytic and so is $\VV(x)^\perp = Q(x) \CC^d$.
 For part (iii) note that $A(x)^{-1} \VV(x)=(A^*(x) \VV(x)^\perp)^\perp$ which combining (i) and (ii) induces an analytic subspace.
 Part (iv) follows from Lemma~\ref{lem-lift}~(ii) applied to a matrix constructed from column vectors of $(V(x),W(x))$ giving a basis of $\VV(x)+\WW(x)$ for almost all $x$. Finally, for part (v) note that $\VV(x)\cap \WW(x) = (\VV(x)^\perp + \WW(x)^\perp)^\perp$ so it follows from (ii) and (iv).
\end{proof}

 \end{lemma}

\subsection{Negative infinite log integral}

\begin{lemma} \label{lem-f-zero}
Let $\XX$ be a compact, connected, analytic manifold (over $\RR$) and $\mu$ a probability measure whose push forward has a continuous density with respect to the Lebesgue measure for any analytic chart of $\XX$.
Suppose that $g\in C^\omega(\XX,\CC)$ and $\int_{\XX}\ln|g(x)|d\mu(x)=-\infty$, then $g=0$, i.e. $g(x)=0$ for all $x\in \XX$. 
\end{lemma}

\begin{proof}
Suppose $g\in C^\omega(\XX,\CC)$ and $g$ is not the zero function. For any $x\in \XX$ such that $g(x)=0$, we claim that there is an open neighborhood $U_x$ such that $\int_{U_x}\ln|g(x)| d\mu(x)>-\infty$. 
Suppose the claim is true, there are finitely many of these open sets $U_i\in \XX$ such that $\bigcup_i U_i\supset \{g(x)=0\}$ (by compactness) 
and for each $i$, $\int_{U_i}\ln|g(x)|\,d\mu(x)>-\infty$. For $x\notin \bigcup_i U_i$ let $|g(x)|>\epsilon$ and $\epsilon<1$, then 
\begin{align*}
&\int_{\XX}\ln|g(x)|d\mu(x)\\
&=\,\int_{\XX\setminus \bigcup_iU_i}\ln|g(x)|\,d\mu(x)\,+\,\int_{\bigcup_iU_i}\ln|g(x)|\,d\mu(x)\\
&\geq\,\ln\epsilon  \,+\,\sum_i \int_{U_i}\ln|g(x)|\,d\mu(x)\,>\,-\infty 
\end{align*}
which contradicts with our assumption.

Now we prove our claim. Let $\dim_\RR \XX=\ell$, $g(x)=0$ and without loss of generality we may use a chart where $x$ is represented by $\nul \in \RR^\ell$.
Using the chart map $\varphi:\XX \to U$ we should technically have $\nul=\varphi(x)$ and work with the functions $g(\varphi^{-1}(x))$ on $U$. But for simplicity we will just write $g(x)$ for $x\in U\subset \RR^\ell$.
Then we have $g(\nul)=0$ and by connectedness of $\XX$, $g$ is not identically zero on this chart. Otherwise, $g$ would be identically zero on $\XX$. 
Moreover, $g(\cdot, 0,\dots,0)$ shall not be the zero function near $0$, otherwise we replace $g$ by $g\circ A$, for $A\in \GL(\ell,\RR)$)\footnote{If for any $A\in \GL(\ell,\RR)$, $g\circ A(\cdot, 0,\dots,0)$ is zero function near $0$, then $g$ must be the zero function.}.
The density of (the push forward by $\varphi$ of) the measure $\mu$ with respect to the Lebesgue measure shall be given by the continuous function $\mu(x)$ near $0$, i.e. $\mu\circ \varphi^{-1}=\mu(x) dx$ represents the measure in the chart.

Then there exist $n\in \NN$ such that for $k<n$, $\frac{\partial^k g}{\partial x_1^k}(0,\dots,0)=0$ and $\frac{\partial^n g}{\partial x_1^n}(0,\dots,0)\neq 0$. By the Weierstrass preparation theorem, on a neighborhood of $\nul=(0,\dots,0)$ we have 
\begin{equation}\label{weier prepa}
g(x_1,x_2,\dots,x_\ell)=W(x_1)h(x_1,x_2,\dots,x_\ell)
\end{equation}
where $h$ is analytic and $h(\nul)=h(0,\dots,0) \neq 0$. $W(x_1)$ is a Weierstrass polynomial, i.e. $$W(x_1)=x_1^{n-1}+g_{n-1}x_1^{n-1}+\dots +g_0$$ where $g_i(x_2,\dots, x_\ell)$ is analytic and $g_i(0,\dots,0)=0$.
Let $r_i(x_2,\dots,x_\ell)$, $i=1,\ldots,n$ be the (possibly complex) roots of $W(x_1)$.
Choose $\delta>0,\,C>0$ such that $|x_i|<\delta$ for all $i$ implies
$$\min(|h(x)|,1)>\frac{|h(\nul)|}{C},\quad \mu(x)<C,\quad \ln|x_1-r_i(x_2,\ldots,x_\ell)|<0\,.$$ 
Then,
\begin{align*}
&\int_{(-\delta,\delta)^\ell} \ln|g(x)| \,\mu(x)\,dx \\
&\geq\, \int_{(-\delta,\delta)^\ell}\left( \ln\frac{|h(\nul)|}{C}\,+ \sum_{i=1}^n  \ln |x_1-r_i(x_2,\dots, x_\ell)| \right) \mu(x)\,dx \\
&\geq\,  C (2\delta)^\ell \ln\frac{|h(\nul)|}{C}  \,+\,C \sum_{i=1}^n \int_{(\delta,\delta)^\ell} \ln|x_1-r_i(x_2,\dots, x_\ell)|\,dx  \\
&\geq\, C\, \left((2\delta)^\ell\ln\frac{|h(\nul)|}{C}  \,+\, n\,(2\delta)^{\ell-1}\min_{r\in \CC}\int_{(-\delta,\delta)}\ln |x_1-r|\,dx_1 \right) 
\,>\,-\infty
\end{align*}
\end{proof}

\begin{coro}\label{cor-f-zero}
 Let $\XX$ be a compact, connected, analytic manifold (over $\RR$) and $\mu$ a probability measure whose push forward has a continuous density for any chart of $\XX$.
 Suppose that $g\in C^\omega(\XX,\CC)$ is not the zero function. Then $\mu\{x:g(x)=0\}=0$.
\end{coro}
\begin{proof}
 Assume $\mu\{x:g(x)=0\}>0$. Then, clearly $\int \ln|g(x)|\,d\mu(x)\,=\,-\infty$ and hence $g=0$ by the lemma above, which contradicts the assumption.
\end{proof}

\end{document}